\documentclass[10pt]{amsart}

\usepackage{lmodern,graphicx}
\usepackage{amssymb}
\usepackage{amsmath}
\usepackage{amsfonts}
\usepackage{amsthm}
\usepackage[all,import]{xy}
\usepackage{url}
\numberwithin{equation}{section}
\usepackage[colorlinks=true, pdfstartview=FitV, linkcolor=black, 
            citecolor=black, urlcolor=black]{hyperref}
\usepackage{bookmark}
\theoremstyle{definition}
\newtheorem{theorem}{Theorem}[]

\newtheorem{lemma}[theorem]{Lemma}

\newtheorem{corollary}[theorem]{Corollary}

\newtheorem{defn}[theorem]{Definition}

\newtheorem{remark}[theorem]{Remark}

\theoremstyle{remark}

\newcommand{\nc}{\newcommand}
\nc{\DMO}{\DeclareMathOperator}	

\nc{\newnotation}{\nomenclature}
\nc{\Cob}{\mathsf{Cob}}
\nc{\fat}{\mathsf{fat}}
\nc{\cob}{\mathsf{Cob}}
\nc{\sets}{\mathsf{Sets}}
\nc{\symp}{\mathsf{Symp}}
\nc{\ssets}{\mathsf{sSets}}
\nc{\cmpct}{\mathsf{cmpct}}
\nc{\pwrap}{\mathsf{PWrap}}
\nc{\coder}{\mathsf{Coder}}
\nc{\grmod}{\mathsf{GrMod}}
\nc{\spaces}{\mathsf{Spaces}}
\nc{\pwrms}{\mathsf{PWrFuk}_{M,S}}
\nc{\fuk}{\mathsf{Fukaya}}
\nc{\fukaya}{\mathsf{Fukaya}}
\nc{\fukml}{\mathsf{Fukaya}_{M,\Lambda}}
\nc{\fukmle}{\mathsf{Fukaya}_{M,\Lambda,\epsilon}}
\nc{\lag}{\mathsf{Lag}}
\nc{\lagml}{\lag_{M,\Lambda}} % For when I get lazy.
\nc{\lagmle}{\lag_{M,\Lambda,\epsilon}}
\nc{\fun}{\mathsf{Fun}}
\nc{\vect}{\mathsf{Vect}}
\nc{\chain}{\mathsf{Chain}}
\nc{\wrfuk}{\mathsf{WrFukaya}}
\nc{\pwrfuk}{\mathsf{PWrFukaya}}
\nc{\inffuk}{\mathsf{InfFuk}}
\nc{\inftycat}{\infty\mathsf{Cat}}
\nc{\corres}{\mathsf{Corres}}
\nc{\cat}{\mathsf{Cat}}
\nc{\fukep}{\fukaya_\Lambda(M,\epsilon)}
\nc{\fukepop}{\fukaya_\Lambda(M,\epsilon)^{\op}}
\nc{\lagep}{\lag_\Lambda(M,\epsilon)}
\DMO{\cyl}{cyl} % Cylindrical
\nc{\dbcoh}{D^b\mathsf{Coh}}
\nc{\corr}{\mathsf{Corr}}

\DMO{\conf}{Conf}
\DMO{\chains}{Chains}
\DMO{\cochains}{Cochains}
\DMO{\cone}{Cone}
\DMO{\ran}{Ran}
\DMO{\leg}{Leg}
\DMO{\cube}{Cube}
\DMO{\floer}{Floer}
\DMO{\maps}{Maps}
\DMO{\exact}{exact}
\DMO{\Decomp}{Decomp}
\DMO{\decomp}{Decomp}
\DMO{\yoneda}{Yoneda}
\DMO{\holomaps}{Holomaps}
\DMO{\comp}{Comp}
\DMO{\crit}{Crit}
\DMO{\test}{{test}}
\DMO{\sign}{sign}
\DMO{\topp}{top}
\DMO{\indx}{Index}
\DMO{\Break}{Break} % Partitions
\DMO{\zero}{zero} %Zero
\DMO{\ob}{Ob}
\DMO{\gr}{Gr} % Grassmanian
\DMO{\Gr}{Gr} % Grassmanian
\DMO{\cl}{Cl} % Clifford Algebra
\DMO{\grlag}{GrLag}
\DMO{\Pin}{Pin}
\DMO{\Graph}{Graph}
\DMO{\pin}{Pin}
\DMO{\gap}{Gap}
\DMO{\Ex}{Ex}
\DMO{\id}{id}
\DMO{\End}{End}
\DMO{\sym}{Sym} 
\DMO{\aut}{Aut}
\DMO{\DK}{DK} %Dold-Kan
\DMO{\poly}{poly} % Polynomial deRham forms
\DMO{\diff}{Diff} 
\DMO{\dist}{dist} %Distance function
\DMO{\coker}{coker} %Cokernel
\nc{\kernel}{\ker} %Kernel
\DMO{\sspan}{span}
\DMO{\hocolim}{hocolim}	
\DMO{\holim}{holim}
\DMO{\sk}{sk}

\nc{\xto}{\xrightarrow}
\nc{\xra}{\xto}
\nc{\tensor}{\otimes}
\nc{\del}{\partial}
\nc{\delbar}{\overline{\del}}
\nc{\dd}{\diamond}
\nc{\tri}{\triangle}
\nc{\bb}{\Box}
\nc{\into}{\hookrightarrow}
\nc{\contains}{\supset}

\nc{\trbar}{\overline{T^*\RR}}
\nc{\tsa}{Ts\cA}
\nc{\tsb}{Ts\cB}

\nc{\vece}{ {\vec \epsilon}}	
\nc{\vecd}{ {\vec \delta}}
\nc{\vt}{ {\vec t}}
\nc{\vx}{ {\vec x}}
\nc{\vs}{ {\vec s}}

\DMO{\op}{op}

\nc{\hiro}{\textcolor{blue}}

\nc{\eqn}{\begin{equation}}
\nc{\eqnd}{\end{equation}}

\def\cA{\mathcal A}\def\cB{\mathcal B}

\def\RR{\mathbb R}

\def\ZZ{\mathbb Z}

\begin{document}
% ------------------- Title and Author -----------------------------
\title{In simply connected cotangent bundles, exact Lagrangian cobordisms are h-cobordisms}

\author{Hiro Lee Tanaka}
\address{One Oxford Street, Cambridge, MA, 02138}
\email{hirolee@math.harvard.edu}
%\subjclass[2010]{Primary }
\keywords{Cobordisms, Fukaya Categories, Mirror Symmetry}

\date{\today}

\begin{abstract}
We show that if Q is simply connected, then every exact Lagrangian cobordism between compact, exact Lagrangians in the cotangent bundle of Q is an h-cobordism. The result follows as a corollary of the Abouzaid-Kragh theorem.
\end{abstract}

\maketitle

%\tableofcontents

%%%%%%%%%%%%%%%%%%%%%%%%%%%%%%%%%%%%%%%%%%	Introduction
%%%%%%%%%%%%%%%%%%%%%%%%%%%%%%%%%%%%%%%%%%	Introduction
%%%%%%%%%%%%%%%%%%%%%%%%%%%%%%%%%%%%%%%%%%	Introduction
\section{Introduction}
We prove:

\begin{theorem}\label{main-theorem}
Let $Q$ be simply connected and smooth. Any compact exact Lagrangian cobordism between compact, exact Lagrangians in $T^*Q$ is an $h$-cobordism.
\end{theorem}

By the $h$-cobordism theorem~\cite{smale-h-cobordism,milnor-h-cobordism}, we have:

\begin{corollary}
If $Q$ is smooth, simply connected, and has $\dim Q \geq 5$, any two compact, exact Lagrangians related by an exact compact cobordism are diffeomorphic.
\end{corollary}

This result is motivated by two topics of interest. The first is the Nearby Lagrangian Conjecture of Arnol'd~\cite{arnold-first-steps}, which conjectures that if $Q$ is a smooth, compact manifold (not necessarily simply connected), then any compact, exact Lagrangian in $T^*Q$ is Hamiltonian-isotopic to the zero section. Theorem~\ref{main-theorem} shows that the classification of Lagrangian cobordism classes in $T^*Q$ provides another strategy for attacking the conjecture when $\dim Q \geq 5$ and $\pi_1 Q = *$. Namely, if there is only one exact, compact cobordism class, then every compact exact Lagrangian is cobordant to the zero section by a cylinder. An unknotting theorem for cylindrical cobordisms in $T^*(Q \times \RR)$ would exhibit the cobordism as a Hamiltonian isotopy.

The second is recent work on categories of Lagrangian cobordisms, as developed in~\cite{nadler-tanaka} and~\cite{biran-cornea,biran-cornea-2}. For a fixed symplectic manifold $M$ satisfying certain monotonicity or convexity conditions, one can define a category whose objects are Lagrangians (equipped with standard Floer-theoretic decorations) and whose morphisms are Lagrangian cobordisms between them (equipped with compatible decorations). In the setting of~\cite{nadler-tanaka}, all objects and cobordisms are {\em exact}.
%, and this is also a setting in which one can construct a functor from the category of cobordisms to the Fukaya category of $M$~\cite{hiro-thesis,hiro-functor}.
The theorem shows one motivation for considering {\em non-compact} cobordisms in this category. (Otherwise, every morphism is homotopically uninteresting.)

%\begin{remark}
%In the monotone case, \cite{biran-cornea-2} constructs a functor to the $T^s$ category of the Fukaya category. The functor of~\cite{hiro-thesis,hiro-functor} is expected to respect exact triangles, so in the exact Lagrangian case, is expected to induce the functor from ~\cite{biran-cornea-2}.
%\end{remark}

Finally, there is another way in which Theorem~\ref{main-theorem} touches both topics. Recall:

\begin{theorem}[Fukaya-Nadler-Seidel-Smith~\cite{nadler,fukaya-seidel-smith}]
Let $Q$ be smooth, compact, and simply connected.
Any compact exact Lagrangian inside $T^*Q$ is equivalent to $Q$ in the Fukaya category of $T^*Q$, with $\ZZ/2\ZZ$ coefficients. If both $Q$ and the Lagrangian are Spin, then the same is true over arbitrary coefficients.
\end{theorem}

It can be shown that compact, exact Lagrangians related by a compact, exact cobordism are equivalent in the Fukaya category~\cite{hiro-thesis, hiro-functor}. \footnote{Over arbitrary coefficients, one needs each Lagrangian to have a Spin structure, and for the cobordism to respect this structure. In characteristic 2, no such structure is needed.} 
Hence the equivalence class of an object in $\fuk(T^*Q)$ is an invariant of an exact Lagrangian's cobordism class. If this is a complete invariant (in the way Stiefel-Whitney numbers classify unoriented cobordism classes) the above theorem of Fukaya-Nadler-Seidel-Smith, together with Theorem~\ref{main-theorem}, would show that any two exact, compact Lagrangians are $h$-cobordant (to the zero section).

\begin{remark}
Finally, we are told that Lara-Simone Suarez has a result for non-simply-connected cobordisms~\cite{suarez}---specifically, that exact, spin, Lagrangian cobordisms for which the collar inclusions induce isomorphisms on $\pi_{1}$ are diffeomorphic to cylinders. Instead of relying on Abouzaid-Kragh's theorem, she utilizes a previous result of Biran and Cornea~\cite{biran-cornea}. This also allows her to consider symplectic manifolds that need not be cotangent bundles. \end{remark}

\subsection{Acknowledgements} We are grateful to Tim Perutz for helpful feedback on this paper. The author was supported by 
a Presidential Fellowship from Northwestern University's Office of the President, 
an NSF Graduate Research Fellowship, 
and a Mathematical Sciences Research Institute Postdoctoral Fellowship.

\section{Recollections}
Recall that $T^*Q$ has a 1-form $\theta_Q = \sum_i p_i dq_i$ whose derivative is symplectic. 

\begin{defn}
We say a Lagrangian submanifold $L \subset T^*Q$ is {\em exact} if it is equipped with a smooth function $f_L: L \to \RR$ for which $df_L = \theta|_L$. We call $f_L$ a primitive for $L$.
\end{defn}

\begin{defn}
A Lagrangian submanifold $W \subset T^*Q \times T^*(0,1)$ is said to be a {\em Lagrangian cobordism} from $L_0$ to $L_1$ if
\[
 W|_{(0,\epsilon)} = L_0 \times (0,\epsilon)
 \qquad
 \text{and}
 \qquad
 W|_{(1-\epsilon,1)} = L_1 \times (1-\epsilon,1) \subset M \times T^*(0,1).
\]
Fixing primitives $f_{L_i}$, we say $W$ is {\em exact} if one can choose a function $f_W: W \to \RR$ so that 
\[
 f_W|_{L_0\times (0,\epsilon)} = f_{L_0}
 \qquad
 f_W|_{L_1 \times (1-\epsilon, 1)} = f_{L_1},
 \qquad
 df_W = (\theta_Q +\theta_{\RR})|_W.
\]
In particular, the value of $f$ is independent of $t$ when $t \in (0,\epsilon) \bigcup (1-\epsilon,1).$
\end{defn}

Finally, recall the following:

\begin{theorem}[Abouzaid-Kragh~\cite{abouzaid-nearby,abouzaid-kragh}]
If $L \subset T^*Q$ is a compact, exact Lagrangian, then the projection map $L \to Q$ is a homotopy equivalence.
\end{theorem}

\section{Proof}
Given a (not necessarily Lagrangian) cobordism $Z$, let $Z^{\op}$ be the same cobordism with ingoing and outgoing boundaries interchanged. We call the composite $Z \circ Z^{\op}$ a {\em double} of $Z$. (The other double is the composition $Z^{\op} \circ Z$.)

\begin{lemma}\label{lemma.double-cylinder-1}
Let $Q$ be simply connected. Then any compact, exact cobordism between two compact exact Lagrangians in $T^*Q$ has a double which is an $h$-cobordism.
\end{lemma}

\begin{lemma}\label{lemma.double-cylinder-2}
Let $Y_0, Y_1$ be simply connected manifolds. If $Z$ is any smooth cobordism between them such that a double is an $h$-cobordism, then $Z$ is an $h$-cobordism itself.
\end{lemma}

\begin{proof}[Proof of Theorem.]
By Abouzaid-Kragh, any two compact exact Lagrangians in $T^*Q$ are homotopy equivalent to $Q$, so are simply connected. By the lemmas, any compact exact cobordism between them is an $h$-cobordism.
\end{proof}

\begin{proof}[Proof of Lemma~\ref{lemma.double-cylinder-1}.]
Let $Y_{01} \subset T^*Q \times T^*(0,1)$ be a cobordism from $Y_0$ to $Y_1$, with each $Y_i \subset T^*Q$ compact and exact. Note that the diffeomorphism $\RR \to \RR$ given by $t \mapsto -t$ gives rise to another exact cobordism $Y_{01}^{\op} \subset M \times T^*(0,1)$, from $Y_1$ to $Y_0$.
We prove that the composite cobordism
\[
N = Y_{01} \circ (Y_{01})^{\op}
\]
is an $h$-cobordism.

Note that since $N$ is a cobordism collared by $Y_1$ on both ends, we can glue the $Y_1$ on both ends via the identity to obtain a compact, exact Lagrangian in $T^*(Q \times S^1)$. (We are guaranteed exactness since, by definition of cobordism, the primitive function $f$ realizing $df = \theta_{Y_1}$ must agree along the collars). We call this Lagrangian $\overline N$.

By Abouzaid-Kragh, the projection map $\overline N \to Q \times S^1$ is a homotopy equivalence. Consider the diagram
\[
\renewcommand{\labelstyle}{\textstyle}
\xymatrix@R=0.5pc@C=0.5pc{
 Y_1 \coprod Y_1 \ar[rr] \ar[dd]^{\rotatebox{90}{$\sim$}} \ar[dr]
 && N \ar[dd] \ar[dr]\\
 & Y_1 \times I \ar[rr] \ar[dd]^{\rotatebox{90}{$\sim$}}
 && \overline{N} \ar[dd]^{\rotatebox{90}{$\sim$}} \\
 Q \coprod Q \ar[rr] \ar[dr] 
 && Q \times I \ar[dr] \\
 & Q \times I \ar[rr]
 && Q \times S^1
 }
\]
where the top and bottom faces are pushout squares. (The all vertical arrows are projection maps to the zero section, and the indicated arrows are equivalences by Abouzaid-Kragh.) By excision and the Five Lemma, the remaining vertical arrow induces an isomorphism in homology $H_*(N) \to H_*(Q \times I)$.

Moreover, $N$ has trivial fundamental group: the groupoid version of van Kampen's theorem shows there must be a pushout diagram of groupoids
\[	
\xymatrix{
\Pi(Y_0 \coprod Y_0) \simeq \ast \coprod \ast \ar[r] \ar[d] 
& B\pi_1 N \ar[d]
\\ \Pi(Y_0 \times [-\epsilon,+\epsilon] \simeq \ast \ar[r]
& \Pi(N) \cong B \ZZ
}
\]
whence it follows that $\ZZ \cong \ZZ \ast \pi_1 N$, hence $\pi_1 N \cong 0$.
By the Whitehead and Hurewicz theorems, the projection map thus realizes a homotopy equivalence $N \simeq Q \times I$. Moreover, the inclusion of $Y_1$ on either collar of $N$ fits into a commutative diagram of topological spaces
\[
\xymatrix{
Y_1 \ar[r] \ar[d]^\pi
&N \ar[d]^\pi \\
Q \ar[r] 
& Q \times [a,b].
}
\]
The lefthand vertical arrow is a homotopy equivalence by the Abouzaid-Kragh theorem, as is the righthand vertical arrow from our previous discussion. The bottom horizontal arrow is obviously a homotopy equivalence, so the top horizontal arrow must be as well. This shows that $N$ is in fact an $h$-cobordism from $Y_1$ to itself.
\end{proof}

\begin{proof}[Proof of Lemma~\ref{lemma.double-cylinder-2}.]
Note we have the inclusion $Y_0 \into Y_{01} \circ (Y_{01})^{\op}$. We make use of the pushout diagram
\[
\xymatrix{
Y_0 \ar[r] \ar[d] 
& Y_{01} \ar[d]
\\ Y_{01}^{\op} \ar[r]
& Y_{01} \circ (Y_{01})^{\op}.
}
\]
Knowing that $Y_0$ is simply connected and that $Y_1 \simeq Y_{01}^{\op} \circ Y_{01}$ is simply connected, the Van Kampen theorem shows that $Y_{01}$ is also simply connected. Finally, the Mayer-Vietoris sequence splits at each level by including $Y_1 \to Y_{01}$. Thus the Whitehead and Hurewicz theorems show that $Y_{01}$ is an $h$-cobordism.
\end{proof}

%%%%%%%%%%%%%%%%%%%%%%%%%%% Bibliography
%%%%%%%%%%%%%%%%%%%%%%%%%%% Bibliography
%%%%%%%%%%%%%%%%%%%%%%%%%%% Bibliography
%%%%%%%%%%%%%%%%%%%%%%%%%%% Bibliography
%%%%%%%%%%%%%%%%%%%%%%%%%%% Bibliography
%%%%%%%%%%%%%%%%%%%%%%%%%%% Bibliography
%%%%%%%%%%%%%%%%%%%%%%%%%%% Bibliography
%%%%%%%%%%%%%%%%%%%%%%%%%%% Bibliography
%%%%%%%%%%%%%%%%%%%%%%%%%%% Bibliography
%%%%%%%%%%%%%%%%%%%%%%%%%%% Bibliography
%%%%%%%%%%%%%%%%%%%%%%%%%%% Bibliography
%%%%%%%%%%%%%%%%%%%%%%%%%%% Bibliography
%%%%%%%%%%%%%%%%%%%%%%%%%%% Bibliography
%%%%%%%%%%%%%%%%%%%%%%%%%%% Bibliography
%%%%%%%%%%%%%%%%%%%%%%%%%%% Bibliography
%%%%%%%%%%%%%%%%%%%%%%%%%%% Bibliography
%%%%%%%%%%%%%%%%%%%%%%%%%%% Bibliography
%%%%%%%%%%%%%%%%%%%%%%%%%%% Bibliography
%%%%%%%%%%%%%%%%%%%%%%%%%%% Bibliography
%%%%%%%%%%%%%%%%%%%%%%%%%%% Bibliography

\bibliographystyle{amsalpha}
\bibliography{../biblio}

%\layout

%How To Compile
%
%The document is test.tex and I compiled it as follows.
%
%xelatex test
%makeindex test
%makeindex test.nlo -s nomencl.ist -o test.nls
%xelatex test

%	How to make new notation index
%   \newnotation{NOTATION}{EXPLANATION}
%	\index{Name}

\end{document}